\documentclass[11pt,a4paper]{article}

\usepackage{amssymb, latexsym, amsthm}
\usepackage{graphicx} 
\usepackage{url}      
\newcommand{\doi}[1]{\url{http://dx.doi.org/#1}}
\usepackage{amsmath}  

\newcommand{\p}{\partial}
\newcommand{\D}{\Delta}
\newcommand{\om}{\omega}

\newcommand{\R}{{\mathbb R}}

\newcommand{\spde}{\textsc{spde}}
\newcommand{\swe}{\textsc{swe}}
\newcommand{\she}{\textsc{she}}
\newcommand{\rds}{\textsc{rds}}
\newcommand{\see}{\textsc{see}}

\renewcommand{\mid}{:}
\let\overline\bar

\newtheorem{theorem}{Theorem}
\newtheorem{lemma}[theorem]{Lemma}
\newtheorem{coro}[theorem]{Corollary}
\newtheorem{definition}[theorem]{Definition}
\newtheorem{assume}[theorem]{Condition}
\theoremstyle{remark}
\newtheorem{remark}{Remark}
\topsep=\parskip

\title{Approximation of the random inertial manifold of singularly perturbed stochastic wave equations}

\author{
Yan Lv\thanks{School of Science, Nanjing University of Science \&
Technology, Nanjing, 210094, \textsc{China}
\protect\url{mailto:lvyan1998@yahoo.com.cn} }
\and 
Wei Wang\thanks{Department of Mathematics, Nanjing
University, Nanjing, 210093, \textsc{China}
\protect\url{mailto:wangweinju@yahoo.com.cn}; and School of
Mathematical Sciences, University of Adelaide, South Australia,
\textsc{Australia}}
\and 
A.~J. Roberts\thanks{School of
Mathematical Sciences, University of Adelaide, South Australia,
\textsc{Australia} \protect\url{mailto:anthony.roberts@adelaide.edu.au}}
}

\date{\today}

\begin{document}

\maketitle

\begin{abstract}
By applying Rohlin's result on the classification of homomorphisms of Lebesgue space, the random inertial manifold of a stochastic damped nonlinear wave equations with singular perturbation is proved to be approximated almost surely by that of a stochastic nonlinear heat equation which is driven by a new Wiener process depending on the singular perturbation parameter.  This approximation can be seen as the  Smolukowski--Kramers approximation as time goes to infinity. However, as time goes infinity,  the approximation changes with the small parameter,  which is different from the approximation  on a finite interval.

\end{abstract}





\section{Introduction}\label{sec:intro}

The motion of particles in a continuum with a stochastic force $\dot{W}$, by Newton's law, is assumed to be described by the following stochastically forced damped wave equation (\swe)~\cite{CF06}
\begin{align}\label{e:SWE1}
&\nu u^{\nu}_{tt}(t,x)+u^{\nu}_{t}(t,x)=\D u^{\nu}+f(u^{\nu}(t,x))+ \sigma W_{t}(t,x),
\quad t>0\,,\  x\in D\,, \\
&u^{\nu}(0,x)=u_{0}\,,\quad u^{\nu}_{t}(0,x)=u_{1}\,,\quad u(t,x)=0\,,\quad x\in \p D\,.\label{e:SWE2}
\end{align}
Here we  consider the problem on a one dimensional bounded spatial domain and for simplicity we assume the domain $D=(0, \pi)$, $W$~is a Wiener process defined on a complete probability space $(\Omega, \mathcal{F}, \{\mathcal{F}_{t}\}_{t\in\R}, \mathbb{P})$ which is determined later.   The small parameter $\nu>0$ characterises the density of particles.   

The  Smolukowski--Kramers approximation of infinite dimension~\cite{CF06} states that on any finite time interval $[0, T]$,  for $0<\nu\ll 1$\,,    the solution $u^{\nu}$ to the \swe~(\ref{e:SWE1})--(\ref{e:SWE2}) is approximated by the solution of the following stochastic nonlinear heat equation (\she)
\begin{eqnarray}\label{e:SHE1}
&&u_{t}(t,x)=\D u(t,x)+f(u(t,x))+\sigma W_{t}(t,x)\,,
\quad x\in D\,,
\\&& u(0,x)=u_{0}\,,\quad u(t,x)=0\,, \quad x\in\p D\,,\label{e:SHE2}
\end{eqnarray}      
in the sense that 
\begin{equation}\label{e:finite-SK}
\lim_{\nu\rightarrow 0}\mathbb{P}\left\{\sup_{0\leq t\leq T}\|u^{\nu}(t)-u(t)\|_{L^{2}(D)}\geq \delta\right\}=0
\end{equation}
for any $\delta>0$\,.  Here we give an almost sure  approximation  for the random dynamics  of  the \swe~(\ref{e:SWE1})--(\ref{e:SWE2}); that is, we consider the approximation of the long time behaviour of~$u^{\nu}$ for small~$\nu$. We call this  the Smolukowski--Kramers approximation  for the \swe~(\ref{e:SWE1})--(\ref{e:SWE2}) as $t\rightarrow\infty$\,.  For this we consider the approximation of random inertial  manifold to \swe~(\ref{e:SWE1}) for small $\nu>0$\,.

%
%
%

Random invariant manifolds are very  important in modelling random dynamics of a stochastic system~\cite[e.g.]{Roberts06k}, especially infinite dimensional systems~\cite[e.g.]{Arn, DuanLuSch03, DuanLuSch04, LuSch07, WangDuan07, Roberts05c}.  For example, Wang and Roberts~\cite{Wang2010a} showed one way to view spatial discretisations of \spde{}s as a stochastic slow invariant manifold.    Duan et al.~\cite{DuanLuSch03, DuanLuSch04} generalized deterministic methods to construct a random invariant manifold for stochastic partial differential equations with multiplicative noise.  Roberts~\cite{Roberts05c} established approximations to stochastic slow invariant manifold models of nonlinear reaction-diffusion \spde{}s. Then some subsequent work constructed random invariant manifolds for a stochastic wave equation~\cite[e.g.]{Liu08, LuSch07}.   We apply  the Lyapunov--Perron method for stochastic partial differential equations~\cite{DuanLuSch04}  to construct a random invariant manifold for the \swe~(\ref{e:SWE1})--(\ref{e:SWE2}) for any fixed $\nu>0$ and a random invariant manifold for the \she~(\ref{e:SHE1})--(\ref{e:SHE2}).   Notice that the noises in systems~(\ref{e:SWE1})--(\ref{e:SWE2}) and~(\ref{e:SHE1})--(\ref{e:SHE2}) are additive: to apply the
Lyapunov--Perron method to \spde{}s with additive noise, we need a stationary solution to transform the \spde{}s to a random differential system~(\ref{e:U-REE})~\cite{DuanLuSch04}.  Then a random invariant manifold  to this stationary solution can be constructed~\cite{DuanLuSch04}.
However,  for a nonlinear stochastic system,  more detailed estimates on solutions is required to ensure the existence of a stationary solution~\cite{PG92, PZ96} and such a stationary solution is difficult to be written out explicitly;  we do this transform by introducing stationary solutions of some linear systems, which are written out explicitly,  see section~\ref{sec:RIM}.  For the \swe~(\ref{e:SWE1})--(\ref{e:SWE2}) we introduce the stationary solution~$z^{*\nu}$ solving the linear system
\begin{equation}\label{e:LSWEs}
\nu z^{\nu}_{tt}+z^{\nu}_{t}=\Delta z^{\nu}+\dot{W}\,.
\end{equation}
and for the \she~(\ref{e:SHE1})--(\ref{e:SHE2}) we introduce $z^{*}$ solving the linear system
\begin{equation}\label{e:eta}
z_{t}=\Delta z+\dot{W}\,.
\end{equation}
Using these stationary processes $z^{*\nu}$ and $z^{*}$,  we transform the \spde{}s to random differential equations and show that this leads to the exact random invariant manifold of the \spde{}s (Theorem~\ref{thm:RIMREE-V}). Such a transformation is frequently invoked in research on \spde{}s~\cite[e.g.]{Liu08, LuSch07}; we verify  rigorously the effectiveness of  this  transformation.

One big  difficulty in  approximating the random invariant manifolds of the \swe~(\ref{e:SWE1})--(\ref{e:SWE2}) by that of the \she~(\ref{e:SHE1})--(\ref{e:SHE2}),  is that  second order derivatives in time of $u^{\nu}$~and~$u$ cannot be treated path-wise in the usual phase space.    The difficulty for~$u^{\nu}$ can be overcame  by  the introduction of~$z^{*\nu}$. However, because $z^{*}_{t}$ cannot be treated as  continuous process, we cannot overcome this difficulty for~$u$  by this transformation.  Fortunately, by Rohlin's classification of homomorphisms on Lebesgue space (Appendix~\ref{sec:Rohlin}), as the distribution of~$z^{*}(\theta_{t}\omega)$ is the same as that of~$z^{*\nu}(\theta_{t}\omega)$ (Appendix~\ref{sec:App-A}), there is a measure preserving mapping $\psi^{\nu}$ (Appendix~\ref{sec:Rohlin}) on the probability space~$(\Omega, \mathcal{F}, \mathbb{P})$  such that 
\begin{equation*}
z^{*}(\psi^{\nu}\theta_{t}\omega)=z^{*\nu}(\theta_{t}\omega).
\end{equation*}
So we can consider $z^{*}(\psi^{\nu}\theta_{t}\omega)$ instead of~$z^{*}(\theta_{t}\omega)$.  Our  result~(Theorem~\ref{thm:main}) on random invariant manifolds implies that the approximate system is 
\begin{equation}\label{e:SHE1}
\tilde{u}^{\nu}_{t}(t,x)=\D \tilde{u}^{\nu}(t,x)+f(\tilde{u}^{\nu}(t,x))+\sigma W^{\nu}_{t}(t,x),
\end{equation}
where the Wiener process $W^{\nu}(t,x)=\psi^{\nu}W(t,x)$.  
This approximating result also  shows that,  different from the Smolukowski--Kramers approximation on finite time interval,  as $t\rightarrow\infty$\,,  for  small $\nu>0$  and almost all $\omega\in\Omega$\,,  the solution $u^{\nu}(t,x, \omega)$ to \textsc{swe}~(\ref{e:SWE1}) is approximated by~$u(t,x, \psi^{\nu}\omega)$, the solution to \textsc{she}~(\ref{e:SHE1}) on the $\psi^{\nu}\omega$~path.    Such transitions of the random parameter~$\omega$ also appears in approximations of the random invariant manifold for slow-fast stochastic system~\cite{WR09}.  However, the transition of~$\psi^{\nu}$ here is difficult to be defined explicitly.  This is left for future research. 

Similar to the analysis of deterministic wave equations~\cite{ChowLu89}, we here introduce the change of variables~(\ref{e:u-v-nu}) and a new inner product on the phase space (section~\ref{sec:RIM-SWE}).  Because of this change of variables,  we restrict the nonlinearity to satisfy  $f(0)=0$\,, which was also needed for the analysis of deterministic wave equations~\cite{ChowLu89}.  Our results generalise the deterministic results~\cite{ChowLu89}.

There has been some research on the approximation of the \swe~(\ref{e:SWE1})--(\ref{e:SWE2}) as $\nu\rightarrow 0$ on finite time intervals~\cite{CF06, CF07}. But there has been little research on the approximation of the long time behaviour  of the \swe~(\ref{e:SWE1})--(\ref{e:SWE2}). However, recent research gave an approximation for  the  long time behaviour in an almost sure sense~\cite{LvWang08} and distribution~\cite{LvWang10} in the special case $\sigma=\sqrt{\nu}$\,.

\section{Preliminary}

Denote by $L^{2}(D)$ the set of square integrable functions on~$(0, \pi)$, and denote by~$\langle \cdot, \cdot\rangle$ the usual inner product, $\|\cdot\|$~the norm on~$L^{2}(D)$.  We also denote by $H_{0}^{1}(D)$ the usual Sobolev space $W^{1, 2}_{0}(D)$~\cite{Adam}.

Let  $A=\D$ with zero Dirichlet boundary condition on $(0, \pi)$. Then the operator~$A$  generates  a strongly continuous semigroup~$e^{At}$, $t\geq 0$\,, on~$L^{2}(D)$. Denote  the eigenvalues of~$-A$ by  $\lambda_{k}=k^{2}$,  $k=1,2, \dots$\,,  and the corresponding  eigenfunctions~$\{e_{k}\}$ which forms a standard orthogonal basis in~$L^{2}(D)$. The nonlinearity $f: \R\rightarrow \R$ is Lipschitz continuous with Lipschitz constant~$L_{f}$, and then   there is a constant $C>0$ such that 
\begin{equation}
|f(\xi)|\leq C(1+|\xi|)\quad \text{for any }\quad \xi\in\R\,. 
\end{equation}
Furthermore, we assume 
\begin{equation}\label{e:Lip}
L_{F}\leq \sqrt{\lambda_{1}}\,.
\end{equation}
The above condition ensures the existence of a unique stationary solution to stochastic wave equations~(\ref{e:SWE1})--(\ref{e:SWE2})~\cite{DaPrato}.

The Wiener process~$\{W(t,x)\}_{t\in\R}$ is assumed to be a two sided, $L^{2}(D)$-valued, $Q$-Wiener process with covariance operator~$Q$ satisfying 
\begin{equation}\label{e:Q}
\operatorname{Tr} Q<\infty.
\end{equation} 
For our purpose we assume the probability space~$(\Omega, \mathcal{F}, \{\mathcal{F}_{t}\}_{t}, \mathbb{P})$ be the canonical probability space with Wiener measure~$\mathbb{P}$~\cite{Arn}. To be more
precise, $W$~is the identity on~$\Omega$, with
\begin{equation*}
\Omega=\left\{w\in C(\R, L^{2}(D)): w(0)=0 \right\}.
\end{equation*}
Let~$\theta_{t}: (\Omega, \mathcal{F}, \{\mathcal{F}_{t}\}_{t\in\R},\mathbb{P})\rightarrow (\Omega, \mathcal{F}, \{\mathcal{F}_{t}\}_{t\in\R},\mathbb{P}) $ be a
metric dynamical system (driven system), that is,
\begin{itemize}
    \item $\theta_0=\text{id}$,
    \item $\theta_t\theta_s=\theta_{t+s}$ for all~$s$,
        $t\in\mathbb{R}$\,,
    \item the map $(t,\omega)\mapsto \theta_t\omega$ is
     measurable and $\theta_t\mathbb{P}=\mathbb{P}$ for all
        $t\in\mathbb{R}$\,.
\end{itemize}
On this canonical probability space~$\Omega$, we choose~$\theta_t$ to be the Wiener shift~\cite{Arn}
\begin{equation}\label{MD}
\theta_t\omega(\cdot)=\omega(\cdot+t)-\omega(t),\quad t\in
\mathbb{R}\,, \quad\omega\in\Omega_{0}\,,
\end{equation}
which preserves the Wiener measure~$\mathbb{P}$ on~$\Omega$\,. Furthermore, $\theta_{t}$~is ergodic under Wiener measure~$\mathbb{P}$. Write~$W(t, x)$ as~$W(t, x, \omega)$ to show the dependence on $\omega\in\Omega$\,, then 
\begin{equation*}
W(\cdot, x, \theta_{t}\omega)=W(\cdot+t, x,  \omega)-W(t, x, \omega).
\end{equation*}
In this view, the stochastic wave equation~(\ref{e:SWE1})--(\ref{e:SWE2}) is driven by~$\theta_{t}$.

\section{Random invariant manifold for stochastic evolutionary equation}\label{sec:RIM}
Random invariant manifold theory for stochastic evolutionary equations~(\textsc{see}s) has been developed in lots of research~\cite[e.g.]{BW10, Cara, DuanLuSch03, DuanLuSch04, Sch11, Arn}. Here  we just recall some basic concepts and results.

Let $H$ be a separable Hilbert space with norm~$\|\cdot\|_H$ and inner
product~$\langle\cdot,\cdot\rangle_H$. We consider a stochastic process~$\{\varphi(t)\}_{t\geq0}$ defined on the probability space~$(\Omega, \mathcal{F}, \{\mathcal{F}_{t}\}_{t}, \mathbb{P})$
\begin{definition}
A stochastic process~$\{\varphi(t)\}_{t\geq0}$ is called a
random dynamical system~(\textsc{rds}) over metric dynamical system~$(\Omega, \mathcal{F}, \{\mathcal{F}_{t}\}_{t}, 
\mathbb{P},\{\theta_t\}_{t})$ if $\varphi$ is
$(\mathcal{B}[0,\infty)\times\mathcal{F}\times\mathcal{B}(H),\mathcal{B}(H))$-measurable
\begin{eqnarray*}
   \varphi: \R^+\times\Omega\times H &\rightarrow& H\\
     (t,\omega,x)&\mapsto& \varphi(t,\omega,x)
\end{eqnarray*}
 and for almost all  $\omega\in \Omega$
\begin{itemize}
    \item $\varphi(0,\omega)=id$ (on $H$);
    \item $\varphi(t+s,\omega,x)=\varphi(t,\theta_s\omega,\varphi(s, \omega,x))$
   for all $t,s \in\R^+$, $x\in H$
     (cocycle property).
\end{itemize}
\end{definition}
If $\varphi(t,\omega, \cdot): H \rightarrow H$ is continuous, $\{\varphi(t)\}_{t\geq 0}$ is called a continuous \textsc{rds}.

\begin{definition} 
A random set~$M(\om)$ is called invariant for \textsc{rds}~$\varphi$ if
\begin{eqnarray*}
\varphi(t,\om, M(\om))\subset M(\theta_t \om),\quad\text{for any } t \geq 0\,.
\end{eqnarray*}
If an invariant set~$M(\om)$ is represented by a Lipschitz or~$C^k$ mapping $h(\cdot,\om) : H_{1} \rightarrow H_{2}$
with $H=H_{1}\oplus H_{2}$
 such that
$M(\om) = \{\xi + h(\xi, \om)\mid \xi\in H_{1}\}$,
then we call~$M(\om)$ a  Lipschitz  or~$C^k$  invariant
manifold. Furthermore, if $H_{1}$~is finite dimensional and
$M(\om)$~attracts exponentially all the orbits of~$\varphi$, then we
call~$M(\om)$ a  random stochastic inertial manifold of~$\varphi$.
\end{definition}

For  our purpose we consider the \textsc{rds} defined by the following abstract evolutionary equation with additive noise
\begin{eqnarray}\label{e:u-SEE}
u_t=Au+F(u)+\sigma\dot{W}\,,\quad u(0)=u_{0}\in H\,.
\end{eqnarray}
Here $F: H\rightarrow H$ is globally Lipschitz continuous with Lipschitz
constant~$L_{F}$, $W$~is an $H$~valued Wiener process with trace class covariation operator on~$H$.   Furthermore 
$A: D(A) \subset H\rightarrow H$ is a linear operator which
generates a strongly continuous semigroup~$\{e^{At}\}_{t\geq 0}$ on~$H$, which can be extended to a group $\{e^{At}\}_{t\in\R}$ on~$H$\,. We assume the  following  exponential dichotomy. 
\begin{assume} \label{ass:3}
With exponents $ \beta<\alpha<0$\,, and bound $K> 0$, 
there exists a continuous
projection~$P$ on~$H$ such that
\begin{enumerate}
  \item $Pe^{At} = e^{At}P$, \; $t\in \R$.
  \item the restriction $e^{At}|_{R(P)}, t \geq 0$, is an isomorphism
   of the range~$R(P)$ of~$P$ onto itself, and we denote~$e^{At}$ for
    $t < 0$ the inverse map;
  \item \begin{itemize}
\item $\|e^{At}Px\|_{H} \leq Ke^{\alpha t}\|x\|_{H}$, $t \leq 0$\,, and \item
$\|e^{At}(I-P)x\|_{H}\leq Ke^{\beta t}\|x\|_{H}$, $t \geq 0$\,.
\end{itemize}
\end{enumerate}
\end{assume}
By the assumption we have the following theorem. 
\begin{theorem}
Assume Condition~\ref{ass:3}, then the \see~(\ref{e:u-SEE}) has a unique stationary solution~$u^{*}(t,\omega)=u^{*}(\theta_{t}\omega)$.  
\end{theorem}
\begin{remark}
There has been lots of research on the existence and uniqueness of stationary solution to stochastic evolutionary equations~\cite[e.g.]{PG92, PZ96,   Hairer}. Here the assumption on $\alpha<0$ is not essential; for $\alpha>0$ the above theorem also holds provided the Lipschitz constant~$L_{F}$ is small enough~\cite{Gaan04}.
\end{remark}

Suppose $u^{*}(t,\omega)=u^{*}(\theta_{t}\omega)$ is a stationary solution of the \see~(\ref{e:u-SEE}). We construct a random invariant manifold to the stationary solution~$u^{*}$.  To do this  we transform the \see~(\ref{e:u-SEE}) to a random dynamical system~\cite{DuanLuSch04}.  Define $U=u-u^{*}$, then 
\begin{eqnarray}\label{e:U-REE}
U_t=AU+F(u)-F(u^{*}(\theta_t\omega)), \quad U(0)=U_{0}\in H\,.
\end{eqnarray}
Notice the above system has a unique stationary solution $U= 0$\,.  
For any $U_{0}\in H$\,, there is a unique solution~$\Phi(t,\omega)U_{0}$ to equation~(\ref{e:U-REE}) and $\{\Phi(t,\omega)\}_{t\geq 0}$ defines a continuous random dynamical system  on~$H$~\cite[e.g.]{DuanLuSch04}. Then by the Lyapunov--Perron method for random evolutionary equations~\cite{DuanLuSch04},  we have the following theorem. 

\begin{theorem}\label{thm:RIMSEE}
   Choose  $\eta<0$ such that  spectral  gap condition 
\begin{eqnarray*}
KL_{F}\left( \frac{1}{\alpha-\eta} +\frac{1}{\eta-\beta}
\right)<1\,,
\end{eqnarray*}
holds,  then there exists a Lipschitz random  invariant manifold  for \textsc{see}~(\ref{e:u-SEE}), which is given by
\begin{eqnarray*}
M(\om) = \{(\xi,  h(\xi,\om))+u^{*}(\omega): \xi\in PH\},
\end{eqnarray*}
where $h : PH \rightarrow QH$ is a Lipschitz continuous mapping with Lipschitz constant~$L_{h}$ and
$h(0)=0$\,. Moreover, if 
\begin{eqnarray}\label{e:gapcondition}
KL_{F}\left( \frac{1}{\alpha-\eta} +\frac{1}{\eta-\beta}
\right)+K^2 L_{h} L_{F}\frac{1}{\alpha-\eta}<1\,,
\end{eqnarray}
then $M(\om)$ is  a random inertial manifold for the \see~(\ref{e:u-SEE}).  Furthermore, if $F\in C^{1}(H, H)$, then the random invariant manifold is also~$C^{1}$, that is, $h\in C^{1}(PH, QH)$.  
 \end{theorem}

 The Lyapunov--Perron method gives an expression of $h: PH\rightarrow QH$ as $h(\xi, \omega)=Q\bar{U}(0,\xi)$ for $\xi\in PH$ with $\bar{U}$~being the unique solution of the following integral equation 
\begin{eqnarray}
\bar{U}(t,\xi)&=&e^{At}(\xi-Pu^{*})+\int_{0}^{t}e^{A(t-s)}P[F(\bar{u}(s))-F(u^{*}(s))]\,ds
\nonumber\\&&{}
+\int_{-\infty}^{t}e^{A(t-s)}Q[F(\bar{u}(s))-F(u^{*}(s))]\,ds\label{e:U-bar}
\end{eqnarray}
in the Banach space 
\begin{eqnarray}\label{e:C-space}
C_{\eta, H}^{-}=\Big\{u \in C((-\infty, 0]; H) : \;   \sup_{t\leq 0}e^{-\eta t}\|u(t)\|<\infty\Big\}
\end{eqnarray}
with norm 
\begin{equation*}
|u|_{C_{\eta, H}^{-}}=\sup_{t\leq 0}e^{-\eta t}\|u(t)\|\,.
\end{equation*}

However, directly constructing an explicit expression to a stationary solution for a nonlinear \spde{} is very difficult. 
So we use another transformation.
Define the stationary process $z^{*}(t,\omega)=z^{*}(\theta_{t}\omega)$ that solves the linear \spde{} 
\begin{equation}
z_{t}=Az+\dot{W}\,.
\end{equation} 
Then $z^{*}(\omega)=\int_{-\infty}^{0}e^{-As}dW(s,\omega)$ and  
\begin{equation}\label{e:z}
z^{*}(t,\omega)=e^{At}\int_{-\infty}^{0}e^{-As}dW(s)+\int_{0}^{t}e^{A(t-s)}dW(s).
\end{equation}
Introduce $V=u-z^{*}$, then 
\begin{eqnarray}\label{e:V-REE}
V_t=AV+F(V+z^{*}(\theta_t\omega)), \quad V(0)=V_{0}\in H\,.
\end{eqnarray}
Then $V^{*}=u^{*}-z^{*}$ is the unique stationary solution to equation~(\ref{e:V-REE}).
Similarly for any $V_{0}\in H$\,, there is a unique solution~$\Psi(t,\omega)V_{0}$ to equation~(\ref{e:V-REE})  and $\{\Psi(t,\omega)\}_{t\geq 0}$~defines a continuous random dynamical systems on~$H$.  By the Lyapunov--Perron method~\cite{DuanLuSch04}, we also have a random invariant manifold~$\widetilde{M}(\omega)$, and then $\widetilde{M}(\omega)+z^{*}(\omega)$ is a random invariant manifold for the \textsc{see}~(\ref{e:u-SEE}). The following theorem establishes that this random invariant manifold coincides with~$ M(\omega)$ in  Theorem~\ref{thm:RIMSEE}.  \begin{theorem}\label{thm:RIMREE-V}
\begin{equation*}
 M(\omega)=\widetilde{M}(\omega)+z^{*}(\omega).
\end{equation*}
\end{theorem}
\begin{proof}
By the Lyapunov--Perron method for a random dynamical system~\cite{DuanLuSch04}, the random dynamical system~$\Psi(t,\omega)$ defined by the random evolutionary equation~(\ref{e:V-REE})  has a random invariant manifold $\widetilde{M}(\omega)=\{(\xi-Pz^{*}(0), \tilde{h}(\xi, \omega)): \xi\in PH\}$ where $\tilde{h}(\xi, \omega)=Q\tilde{V}(0, \xi)$ and where $\tilde{V}$ is the unique solution of the integral equation
\begin{eqnarray*}
\tilde{V}(t,\xi)&=&e^{At}(\xi-Pz^{*}(0))+\int_{0}^{t}e^{A(t-s)}PF(\tilde{V}(s)+z^{*}(s))\,ds\\
&&{}+\int_{-\infty}^{t}e^{A(t-s)}QF(\tilde{V}(s)+z^{*}(s))\,ds
\end{eqnarray*}
in space~$C_{\eta, H}^{-}$.  Since the stationary solution~$V^{*}(\omega)$ lies on this random  invariant manifold, choosing  $\xi=Pu^{*}(0)$, we have $\tilde{V}(t,Pu^{*}(0))=V^{*}(t, \omega)$. Then
  by the expression for~$z^{*}$, 
\begin{eqnarray*}
u^{*}(t,\omega)&=&\tilde{V}(t,Pu^{*}(0))+z^{*}(t)\\
&=&e^{At}Pu^{*}(\omega)+\int_{0}^{t}e^{A(t-s)}PF(u^{*}(s))\,ds+\int_{0}^{t}e^{A(t-s)}PdW(s)\\&&{}+\int_{-\infty}^{t}e^{A(t-s)}QF(u^{*}(s))\,ds+\int_{-\infty}^{t}e^{A(t-s)}QdW(s)
\end{eqnarray*}
and 
\begin{equation}\label{e:u*}
u^{*}(\omega)=u^{*}(0, \omega)=
 Pu^{*}(\omega)+\int_{-\infty}^{0}e^{-As}QF(u^{*}(s))\,ds+\int_{-\infty}^{0}e^{-As}QdW(s).
\end{equation}

Notice that by the integral equation~(\ref{e:U-bar}),
 the solution~$\bar{u}$ to the \see~(\ref{e:u-SEE})  with initial value $\bar{u}(0)=(\xi, h(\xi, \omega))+u^{*}(\omega)$ is 
\begin{eqnarray*}
\bar{u}(t)&=&\bar{U}(t)+u^{*}(t)\\
&=&e^{At}\xi+\int_{0}^{t}e^{A(t-s)}PF(\bar{u}(s))\,ds+\int_{-\infty}^{t}e^{A(t-s)}QF(\bar{u}(s))\,ds+u^{*}(t)\\&&{}-e^{At}Pu^{*}-\int_{0}^{t}e^{A(t-s)}PF(u^{*}(s))\,ds-\int_{-\infty}^{t}e^{A(t-s)}QF(u^{*}(s))\,ds\,.
\end{eqnarray*}
Rewrite  the last three terms in the above equality 
and by~(\ref{e:u*})
\begin{eqnarray*}
&&e^{At}\left[Pu^{*}+\int_{-\infty}^{0}e^{-As}QF(u^{*}(s))\,ds+\int_{-\infty}^{0}e^{-As}QdW(s)\right]
\\&&{}+\int_{0}^{t}e^{A(t-s)}F(u^{*}(s))\,ds+\int_{0}^{t}e^{A(t-s)}dW(s)\\&&{}-\int_{-\infty}^{t}e^{A(t-s)}QdW(s)-\int_{0}^{t}e^{A(t-s)}PdW(s)\\&=&u^{*}(t,\omega)-\int_{-\infty}^{t}e^{A(t-s)}QdW(s)-\int_{0}^{t}e^{A(t-s)}PdW(s).
\end{eqnarray*}
Then we have 
\begin{eqnarray*}
\bar{u}(t)&=&e^{At}\xi+\int_{0}^{t}e^{A(t-s)}PF(\bar{u}(s))\,ds+\int_{-\infty}^{t}e^{A(t-s)}QF(\bar{u}(s))\,ds\\&&{}+\int_{-\infty}^{t}e^{A(t-s)}QdW(s)+\int_{0}^{t}e^{A(t-s)}PdW(s)\\
&=&\tilde{V}(t,\xi)+z^{*}(t).
\end{eqnarray*}
The proof is complete.
\end{proof}

The above theorem shows that if the \see~(\ref{e:u-SEE}) has a unique stationary solution~$u^{*}$,  then the random invariant manifold~$\mathcal{M}(\omega)$ to the stationary solution~$u^{*}$ can be derived from the random invariant manifold~$\widetilde{\mathcal{M}}(\omega)$ to~$V^{*}$, the stationary solution of~(\ref{e:V-REE}), by  
 the transformation $V=u-z^{*}$.

%
%
%
%
 
\section{Random invariant manifold for SWEs}\label{sec:RIM-SWE}
We construct a random inertial  invariant manifold for the \textsc{swe}~(\ref{e:SWE1})--(\ref{e:SWE2}) with fixed parameter $\nu>0$\,.

By the result of   Wang and Lv~\cite{LvWang10},  there is a stationary solution $(u^{*\nu}, u^{*\nu}_{t})\in H_{0}^{1}(D)\times L^{2}(D)$. Furthermore, this stationary solution is unique provided the Lipschitz inequality~(\ref{e:Lip}) holds~\cite{DaPrato}. 
 By the discussion at the end of the last section, we 
 use the transformation 
 $\bar{u}^{\nu}=u^{\nu}-z^{*\nu}$, and for technical reasons we make the change of variables
\begin{equation}\label{e:u-v-nu}
\bar{u}^{\nu}_{t}=-\frac{1}{2\nu}\bar{u}^{\nu}+\frac{1}{\nu}\bar{v}^{\nu}\quad \text{and} \quad \bar{U}^{\nu}=(\bar{u}^{\nu}, \bar{v}^{\nu}).
\end{equation}
The above change of variables is similar to that for the deterministic wave equation~\cite{ChowLu89} and to that in previous research on stochastic wave  equations~\cite{Liu08, LuSch07}.   For convenience, we give a simple description of the construction of the random inertial  manifold. By the definition of~$\bar{U}^{\nu}$ we have  a random differential equation
\begin{equation}\label{e:U-nu-REE}
 \bar{U}_{t}^{\nu}(t,\omega)=C\bar{U}^{\nu}(t,\omega)+F(\bar{U}^{\nu}(t,\omega), \theta_{t}\omega)
\end{equation}
where
\begin{eqnarray*}C=
\begin{bmatrix}
        -\frac{1}{2\nu} & \frac{1}{\nu} \\
        \frac{1}{4\nu}+A & -\frac{1}{2\nu} \\
      \end{bmatrix}, \quad F(\bar{U}^{\nu},\omega)=
      \begin{bmatrix}
        0 \\
        f(\bar{u}^{\nu}+z^{*\nu})\\
      \end{bmatrix}.
\end{eqnarray*}
   We apply Theorem~\ref{thm:RIMSEE} and Theorem~\ref{thm:RIMREE-V} to construct a random invariant manifold for equation~(\ref{e:U-nu-REE}) based upon a stationary solution~$(u^{*\nu}, u^{*\nu}_{t})$.

We first state some facts on the linear operator~$C$.
 Let $E=H_0^1(D)\times L^2(D)$ and $N>0$ be an integer. Set
\begin{eqnarray*}
E_{11}&=&\operatorname{span}\begin{Bmatrix} \begin{bmatrix}
                       e_{k} \\
                       0 \\
                     \end{bmatrix},
                     \begin{bmatrix}
                       0 \\
                      e_{k} \\
                     \end{bmatrix}:
k=1,\ldots,N \end{Bmatrix}\\
 E_{22}&=& \operatorname{span}
 \begin{Bmatrix} \begin{bmatrix}
                       e_{k} \\
                       0 \\
                     \end{bmatrix},
                     \begin{bmatrix}
                       0 \\
                       e_{k} \\
                     \end{bmatrix}:
k=N+1,N+2,\ldots \end{Bmatrix}.
\end{eqnarray*}
It is evident that $E=E_{11}\oplus E_{22}$, that $E_{11}$~is orthogonal to~$E_{22}$
by the orthogonality of~$\{e_{k}\}$, and that $\dim E_{11}=2N $\,. Moreover,
both $E_{11}$~and~$E_{22}$ are invariant subspaces of the operator~$C$.

Since the eigenvalues of~$A$ are~$-k^2$ with corresponding
eigenvectors~$e_{k}$, $k=1,2,\ldots$\,, by restricting~$C$ to~$E_{11}$, the eigenvalues of~$C|_{E_{11}}$ are
\begin{eqnarray*}
\lambda_k^{\pm}=\frac{-1\pm\sqrt{1-4\nu
k^2}}{2\nu}\,, \quad k=1,2,\ldots,N\,,
\end{eqnarray*}
with corresponding eigenvectors
\begin{eqnarray*}
e_k^{\pm}=\begin{bmatrix}
            e_{k}\\
            \pm\frac{\sqrt{1-4\nu k^2}}{2} e_{k}\\
          \end{bmatrix}, \quad k=1,2,\ldots, N\,.
\end{eqnarray*}
Let
\begin{eqnarray*}
E_1= \operatorname{span}\{e_k^+:k=1,\ldots,N\}, \quad 
E_{-1}= \operatorname{span}\{e_k^-:k=1,\ldots,N\}.
\end{eqnarray*}
By this definition $E_{11}=E_1\oplus E_{-1}$, and $E_1$~and~$E_{-1}$ are
invariant subspaces of the operator~$C$. Let $P_1$~and~$P_{-1}$ be
the corresponding spectral projections~\cite{Pazy83} and $P_{22}$ be
the unique orthogonal projection onto~$E_{22}$. Then there exist a
decomposition $E=E_1\oplus E_{-1}\oplus E_{22}$ with projections~$P_1,
P_{-1}, P_{22}$ respectively. Note that $E_1$ is not orthogonal to~$E_{-1}$. To overcome this we invoke an equivalent inner product on~$E$, as defined for the deterministic wave equations~\cite{Mora87}, to ensure $E_{1}$~is orthogonal to~$E_{-1}$.

Let
$U_i=(u_i,v_i)$, $i=1,2$\,, be two elements of~$E$ or $E_{11}$,
$E_{22}$. Assume $\frac{1}{2\sqrt{\nu}}>N+1$, and define the new inner
products on~$E_{11}$ and~$E_{22}$ as 
\begin{eqnarray*}
\langle U_1,U_2 \rangle_{E_{11}}&=&\langle\nu
Au_1,u_2\rangle+\frac{1}{4}\langle u_1,u_2\rangle+\langle v_1,
v_2\rangle,\\
\langle U_1,U_2 \rangle_{E_{22}}&=&\langle
-Au_1,u_2\rangle+\left(\frac{1}{4\nu}-2(N+1)^2\right)\langle
u_1,u_2\rangle+\langle v_1, v_2\rangle,
\end{eqnarray*}
where $\langle\cdot,\cdot\rangle$ is the usual inner product of~$L^2(D)$. Define the new inner product on~$E$ by
\begin{eqnarray*}
\langle U, V\rangle_E=\langle U_{11},V_{11}\rangle_{E_{11}}+\langle
U_{22},V_{22}\rangle_{E_{22}}
\end{eqnarray*}
where $U=U_{11}+U_{22}$ and $V=V_{11}+V_{22}$ with $U_{ii},
V_{ii}\in E_{ii}$, $i=1, 2$\,. The corresponding norm is denoted by~$\|\cdot\|_E$.

 Since $\frac{1}{2\sqrt{\nu}}>N+1$\,, 
$\langle \cdot,\cdot\rangle_{ E_{11}}$ is equivalent to the usual
inner product on~$E_{11}$, and $\langle \cdot,\cdot\rangle
_{E_{22}}$ is equivalent to the usual inner product on~$E_{22}$.
Hence the new inner product~$\langle \cdot,\cdot\rangle_E$ is
equivalent to the usual inner product on~$E$~\cite{Mora87}.

In terms of this new inner product, by the orthogonality of~$\sin kx$, direct methods 
 verify  that $E_{-1}\perp E_{22}$ and $E_1\perp E_{22}$.
Moreover, $E_1\perp E_{-1}$. Let $E_2=E_{-1}\oplus E_{22}$, then
$E_1\perp E_2$.

By the definition of the new inner product, for 
$U=(0,v)\in E$\,,
\begin{eqnarray}\label{norm1}
\|U\|_E=\|v\|_{L^{2}(D)},
\end{eqnarray}
and for any $U=(u,v)\in E$\,, 
 \begin{eqnarray}\label{norm2}
\|U\|_E\geq\sqrt{\frac{1}{4}-\nu(N+1)^2}\|u\|_{L^{2}(D)}\,.
\end{eqnarray}

Let $C_1,C_2, C_{-1}, C_{22}$ denote $C|_{E_1}, C|_{E_2},
C|_{E_{-1}}, C|_{E_{22}}$, respectively. Then similar to Mora's bounds~\cite{Mora87},
\begin{align}
\|e^{C_1t}\|&\leq e^{\lambda_N^+t}\quad\text{for }t\leq 0\,,\label{e:1}\\
\|e^{C_{-1}t}\|&\leq e^{\lambda_N^-t}\quad\text{for }t\geq 0\,,\label{e:2}\\
\|e^{C_{22}t}\|&\leq e^{\lambda_{N+1}^+t}\quad\text{for }t\geq
0\,.\label{e:3}
\end{align}
By the bounds~(\ref{e:2}) and~(\ref{e:3}), we have
\begin{eqnarray*}
\|e^{C_2t}\|\leq e^{\lambda_{N+1}^+t}\quad\text{for }t\geq 0\,.
\end{eqnarray*}
For the nonlinearity~$F$, in terms of the new norm,  by~(\ref{norm1}) and~(\ref{norm2}), 
\begin{eqnarray*}
\|F(\bar{U}_1,\omega)-F(\bar{U}_2,\omega)\|_E
&=&\|f(\bar{u}_1+z^{*\nu})-f(\bar{u}_2+z^{*\nu})\|_{L^{2}(D)}\\
&\leq&L_f\|\bar{u}_1-\bar{u}_2\|_{L^{2}(D)}\\
&\leq&\frac{L_f}{\sqrt{\frac{1}{4}-\nu(N+1)^2}}
\|\bar{U}_1-\bar{U}_2\|_E\\
&\leq&3L_f\|\bar{U}_1-\bar{U}_2\|_E.
\end{eqnarray*}
So $F$ is Lipschitz with respect to~$\bar{U}$ and the Lipschitz constant is independent of~$\nu$ provided  the parameter~$\nu$ is small.

Notice that by choosing $\alpha=\lambda_{N}^{+}$\,, $\beta=\lambda_{N+1}^{+}$ and $\eta=(\lambda_{N}^{+}+\lambda_{N+1}^{+})/2$\,, for $\nu>0$ small enough, the gap condition~(\ref{e:gapcondition}) in Theorem~\ref{thm:RIMSEE} holds. Then a similar  discussion to that by Liu~\cite{Liu08} and Lu \& Schmalfu\ss~\cite{LuSch07} leads to the following theorem.
\begin{theorem}\label{thm:RIMSWE}
There exists $\nu_{0}>0$ such that for any $\nu\in  (0, \nu_{0})$, there is an $N$-dimensional  inertial manifold~$\overline{\mathcal{M}}^{\nu}_{E}(\omega)$ for equation~(\ref{e:U-nu-REE}), which is represented by
\begin{equation*}
 \overline{\mathcal{M}}^{\nu}_{E}(\omega)=\{(\xi, h^{\nu}(\xi, \omega)): \xi\in E_{1}\}
\end{equation*}
with 
\begin{equation*}
h^{\nu}(\cdot,  \omega): E_{1}\rightarrow E_{2}
\end{equation*}
being Lipschitz continuous. 
Moreover, if $f\in C^{1}(L^{2}(D), L^{2}(D))$, then the random invariant manifold is~$C^{1}$, that is $h^{\nu}\in C^{1}(E_{1}, E_{2})$.
\end{theorem}

For our purposes we need some estimates of the  solution on the random  invariant manifold~$\overline{\mathcal{M}}^{\nu}_{E}(\omega)$.  For $\bar{U}^{\nu}_{0}=(\xi,  h^{\nu}(\xi, \omega))\in \overline{\mathcal{M}}^{\nu}_{E}(\omega)$, by the invariance of~$\overline{\mathcal{M}}^{\nu}_{E}(\omega)$,  $\bar{U}^{\nu}(t,\omega)$, the solution of~(\ref{e:U-nu-REE}) with $\bar{U}^{\nu}(0)=\bar{U}^{\nu}_{0}$, lies on~$\overline{\mathcal{M}}^{\nu}_{E}(\theta_{t}\omega)$. Then, by the construction of the random invariant manifold and  the uniqueness of solutions, for~$t\leq 0$
\begin{eqnarray*} 
\bar{U}^{\nu}(t, \omega)&=&e^{Ct}\xi+\int_{0}^{t}P_{1}e^{C(t-s)}F(\bar{U}^{\nu}(s, \omega),  \omega)\,ds\\&&{}+\int_{-\infty}^{t}(P_{-1}+P_{22})e^{C(t-s)}F(\bar{U}^{\nu}(s, \omega), \omega)\,ds\,.
\end{eqnarray*}
Notice that 
\begin{equation*}
\|F(\bar{U}^{\nu})\|_{E} \leq \|f(\bar{u}^{\nu}+z^{*\nu})\|_{L^{2}(D)}\leq 3L_{f}(\|\bar{U}^{\nu}\|_{E}+\|z^{*\nu}\|+1).
\end{equation*}
Then, by the gap condition~(\ref{e:gapcondition}), a direct calculation yields  
\begin{equation}\label{e:bar-U-bound}
|\bar{U}^{\nu}|_{C_{\eta, E}^{-}}\leq  R_{1}(\omega)
\end{equation}
for some tempered random variable~$R_{1}$. Here the Banach space~$C_{\eta, E}^{-}$~is defined by~(\ref{e:C-space}) through replacing~$H$ by~$E$.
Next we need the same estimate on~$\bar{U}^{\nu}_{t}$. Since $\bar{U}^{\nu}(t, \omega)$ lies on~$\overline{\mathcal{M}}^{\nu}(\omega)$, we have 
\begin{equation*}
\bar{U}^{\nu}(t,\omega)=\bar{U}_{N}^{\nu}(t,\omega)+h^{\nu}(\bar{U}_{N}^{\nu}(t,\omega), \omega)
\end{equation*}
 with $\bar{U}_{N}^{\nu}(0,\omega)=\xi$ and $\bar{U}_{N}^{\nu}(t,\omega)\in P_{1}E$ for $t\in\R$\,. 
 Moreover,
 \begin{equation*}
 \dot{\bar{U}}_{N}^{\nu}(t,\omega)
 =C_{1}\bar{U}_{N}^{\nu}(t,\omega)+P_{1}F(\bar{U}_{N}^{\nu}(t,\omega)+h^{\nu}(\bar{U}_{N}^{\nu}(t,\omega), \omega)).
 \end{equation*}
 Then by~(\ref{e:bar-U-bound}), (\ref{e:1}), the spectrum gap condition~(\ref{e:gapcondition}) and the Lipschitz property of~$h^{\nu}$,
 \begin{equation}\label{e:bar-U-N-bound}
 |\dot{\bar{U}}_{N}^{\nu}(t,\omega)|_{C^{-}_{\eta, E}}\leq R'_{2}(\omega).
 \end{equation}
 Notice that 
 \begin{equation*}
 \dot{\bar{U}}^{\nu}(t,\omega)=\dot{\bar{U}}_{N}^{\nu}(t,\omega)+Dh^{\nu}(\bar{U}_{N}^{\nu}(t,\omega), \omega)\dot{\bar{U}}^{\nu}_{N}(t,\omega),
 \end{equation*}
 then by the bound~(\ref{e:bar-U-N-bound}), for some tempered random variable~$R_{2}$, 
 \begin{equation}\label{e:bar-d-U-bounded}
 |\dot{\bar{U}}^{\nu}(t,\omega)|_{C^{-}_{\eta, E}}\leq R_{2}(\omega).
 \end{equation}


\section{Approximation of random inertial manifold}
 
This section addresses the approximation  of~$\overline{\mathcal{M}}_{E}^{\nu}$ for small $\nu> 0$\,.  First we consider the stochastic heat equation~(\ref{e:SHE1})--(\ref{e:SHE2}). Recall the stationary process~$z^{*}$ that solves~(\ref{e:eta}). 
We make the transformation $\tilde{u}=u-z^{*}$,  and derive $\tilde u$ satisfies the \rds
\begin{equation}\label{e:RHE}
\tilde{u}_{t}(t,\omega)=\Delta \tilde{u}(t,\omega)+f(\tilde{u}(t,\omega)+z^{*}(\theta_{t}\omega)).
\end{equation}
Notice that under our assumptions, the stochastic nonlinear heat equation~(\ref{e:SHE1})--(\ref{e:SHE2}) has a unique stationary solution. By  Theorem~\ref{thm:RIMSEE} the following theorem holds. 
\begin{theorem}\label{thm:RIMSHE}
Assume $f\in C^{1}(L^{2}(D), L^{2}(D))$ and $N>0$ large enough   
Then the random equation~(\ref{e:RHE})  has an $N$-dimensional $C^{1}$ random inertial manifold~$\widetilde{\mathcal{M}}_{L^{2}(D)}(\omega)$  with 
\begin{equation*}
\widetilde{\mathcal{M}}_{L^{2}(D)}(\omega)=\left\{(\zeta, h(\zeta, \omega)): \zeta\in P_{N}L^{2}(D)\right\}
\end{equation*}  
where $P_{N}$ is the orthogonal projection from~$L^{2}(D)$ to the  $N$-dimensional space $\operatorname{span}\{e_{1}, e_{2}, \ldots, e_{N}\}$.
\end{theorem} 


Now we define the following  random  set  in~$E$
\begin{equation}\label{M_e}
\widetilde{\mathcal{M}}_{E}(\omega)=\left\{(\tilde{u}_{0}, \tilde{u}_{t}(0, \tilde{u}_{0})): \tilde{u}_{0}\in\widetilde{\mathcal{M}}_{L^{2}(D)}(\omega) \right\},
\end{equation}
and bounded random set
\begin{equation}\label{e:M_eR}
\widetilde{\mathcal{M}}_{ER}(\omega)=\left\{(\tilde{u}_{0}, \tilde{u}_{t}(0, \tilde{u}_{0})) :  \tilde{u}_{0}=\zeta+h(\zeta, \omega)\in
\widetilde{\mathcal{M}}_{L^{2}(D)}(\omega), \|\zeta\|_{L^{2}(D)}<R
\right\},
\end{equation}
where $\tilde{u}(t,\tilde{u}_{0})$ is the unique solution of the \rds~(\ref{e:RHE}) with initial value $\tilde{u}(0, \tilde{u}_{0})=\tilde{u}_{0}$\,,  and where $R>0$ is an arbitrary constant.

As mentioned in the Introduction, we have to  consider the \rds~(\ref{e:RHE}) on~$\psi^{\nu}\omega$.  For this we define $\tilde{u}^{\nu}(t,\omega)$ on the fiber~$\psi^{\nu}\omega$  solving 
\begin{equation}\label{e:RHE-u-tilde}
\tilde{u}^{\nu}_{t}(t,\omega)=\D \tilde{u}^{\nu}(t,\omega)+f(\tilde{u}^{\nu}(t,\omega)+z^{*}(\psi^{\nu}\theta_{t}\omega)).
\end{equation}
Then by the Lyapunov--Perron method we  have an $N$-dimensional  random invariant manifold 
which is exactly~$\widetilde{\mathcal{M}}_{L^{2}(D)}(\psi^{\nu}\omega)$.  
 We give a relation between~$\overline{M}^{\nu}_{E}(\omega)$ and~$\widetilde{\mathcal{M}}_{ER}(\psi^{\nu}\omega)$. 

We  need some estimates of the solution of equation~(\ref{e:RHE-u-tilde}) on~$\widetilde{\mathcal{M}}_{L^{2}(D)}(\psi^{\nu}\omega)$.  For $\tilde{u}_{0}=(\zeta, h(\zeta, \psi^{\nu}\omega))\in\widetilde{\mathcal{M}}_{L^{2}(D)}(\psi^{\nu}\omega)$,  then $\tilde{u}^{\nu}(t,\omega)$, the solution to equation~(\ref{e:RHE-u-tilde}) with $\tilde{u}^{\nu}(0, \omega)=\tilde{u}_{0}$\,, 
 by the invariance of~$\widetilde{\mathcal{M}}_{L^{2}(D)}(\psi^{\nu}\omega)$,   lies on~$\widetilde{\mathcal{M}}_{L^{2}(D)}(\psi^{\nu}\theta_{t}\omega)$.  Then, by a similar discussion for~$\bar{U}^{\nu}$, for $t\leq 0$
\begin{eqnarray*}
\tilde{u}^{\nu}(t,\omega)&=&e^{A t}\zeta+\int_{0}^{t}P_{N}e^{A(t-s)} f(\tilde{u}^{\nu}(s, \omega)+z^{*\nu}(s,\omega))\,ds
\\&&{}+\int_{-\infty}^{t}(\operatorname{Id}-P_{N})e^{A(t-s)}f(\tilde{u}^{\nu}(s, \omega)+z^{*\nu}(s, \omega))\,ds\,.
\end{eqnarray*}
By the gap condition and the Lipschitz property of~$f$,  for some tempered random variable~$R_{3}$, 
\begin{equation*}
|\tilde{u}^{\nu}|_{C^{-}_{\eta, L^{2}(D)}}\leq R_{3}(\omega).
\end{equation*}
Further, we need the following estimate  of~$\tilde{u}^{\nu}_{tt}$  with~$\tilde{u}^{\nu}(t,\omega)$ lying  on the random invariant manifold~$\widetilde{\mathcal{M}}_{ER}(\psi^{\nu}\theta_{t}\omega)$.

%
\begin{lemma}\label{lem:u-tt}
Assume the conditions of Theorem~\ref{thm:RIMSHE}. For each $R>0$\,,  such that for $\|\zeta\|_{L^{2}(D)}\leq R$ and $\zeta\in P_{N}L^{2}(D)$, then almost surely
\begin{equation*}
\nu\|e^{-\eta t}\tilde{u}^{\nu}_{tt}(t,\zeta+h(\zeta, \psi^{\nu}\omega),    \omega) \|_{L^{2}(D)}\rightarrow 0\,,\quad  t\leq 0\,,
\end{equation*}
where $\tilde{u}^{\nu}$ is the unique solution of the \rds~(\ref{e:RHE}) with $\tilde{u}^{\nu}(0)=\zeta+h(\zeta, \psi^{\nu}\omega)$.
\end{lemma}
\begin{proof}
By  the invariance of~$\widetilde{\mathcal{M}}_{L^{2}(D)}(\psi^{\nu}\omega)$ for $\tilde{u}_{0}=\zeta+h(\zeta, \psi^{\nu}\omega)$ 
\begin{equation*}
\tilde{u}^{\nu}(t, \tilde{u}_{0}, \omega)=\tilde{u}^{\nu}_{N}(t,\zeta, \omega)+h(\tilde{u}^{\nu}_{N}(t,\zeta,\omega), \psi^{\nu}\omega)
\end{equation*}
with 
\begin{equation}\label{e:u_N}
\tilde{u}^{\nu}_{N,t}=\Delta\tilde{u}^{\nu}_{N}+P_{N}f(\tilde{u}^{\nu}_{N}+h(\tilde{u}^{\nu}_{N},  \psi^{\nu}\omega)+z^{*\nu}(t, \omega)),\quad \tilde{u}^{\nu}_{N}(0)=\zeta\,.
\end{equation}
Here we use the equality $z^{*}(\psi^{\nu}\theta_{t}\omega)=z^{*\nu}(\theta_{t}\omega)=z^{*\nu}(t,\omega) $.
By a similar discussion to that for the Lyapunov--Perron method to construct a random invariant manifold, we can construct a unique solution~$\tilde{u}_{N}$ to~(\ref{e:u_N}) which is in the space~$C^{-}_{\eta, P_{N}L^{2}(D)}$.  Then 
\begin{equation*}
\|e^{-\eta t}\tilde{u}^{\nu}_{N}(t,   \omega)\|_{L^{2}(D)}\leq C(\omega),\quad t\leq 0\,,
\end{equation*}
for some random constant~$C(\omega)$ which is independent of parameter~$\nu$.
By Theorem~\ref{thm:RIMSHE},  $h\in C^{1}$ and Lipschitz, we have 
\begin{equation}
\nu\|\tilde{u}^{\nu}_{t}\|_{L^{2}(D)}=\nu\|\tilde{u}^{\nu}_{N,t}\|_{L^{2}(D)}+\nu L_{h}\|\tilde{u}^{\nu}_{N,t}\|_{L^{2}(D)}\,.
\end{equation} 
Then by~(\ref{e:u_N}) and above estimate for~$e^{-\eta t}\tilde{u}^{\nu}_{N}$, 
\begin{equation*}
\nu\|e^{-\eta t}\tilde{u}^{\nu}_{t}(t,  \omega)\|_{L^{2}(D)}\rightarrow 0\quad \text{almost surely for any } t\leq 0\,, \quad \text{as}\quad \nu\rightarrow 0\,.
\end{equation*}

Now let $\tilde{w}^{\nu}=\nu\tilde{u}^{\nu}_{t}$\,, then $\tilde{w}^{\nu}\in C_{\eta, L^{2}(D)}^{-}$ and
\begin{equation*}
\tilde{w}^{\nu}_{t}(t, \omega)=A^{\nu}(t,  \omega)\tilde{w}^{\nu}(t, \omega)+\nu B^{\nu}(t,  \omega),\quad \tilde{w}^{\nu}(0)=\nu\tilde{u}^{\nu}_{t}(0),
\end{equation*}
with 
\begin{eqnarray*}
A^{\nu}(t, \omega)&=&\Delta+Df(\tilde{u}^{\nu}(t, \omega)+z^{*\nu}(t, \omega)),\\
B^{\nu}(t, \omega)&=&Df(\tilde{u}^{\nu}(t, \omega)+z^{*\nu}(t, \omega))z_{t}^{*\nu}(t, \omega).
\end{eqnarray*}
 Under the assumptions of Theorem~\ref{thm:RIMSHE},
$\tilde{w}^{\nu}\in C_{\eta, L^{2}(D)}^{-}$ is equivalent to the statement that  $\tilde{w}^{\nu}$ has the following form for $t\leq 0$\,,
\begin{eqnarray*}
\tilde{w}^{\nu}(t, \omega)&=&S^{\nu}(t, \omega)P_{N}\tilde{w}^{\nu}(0)+\nu\int_{0}^{t}S^{\nu}(t-s,  \omega)P_{N}B^{\nu}(s,  \omega)\,ds\\
&&{}+\nu\int_{-\infty}^{t}S^{\nu}(t-s,  \psi^{\nu}\omega)(I-P_{N})B^{\nu}(s,  \omega)\,ds
\end{eqnarray*}
where 
\begin{equation*}
S^{\nu}(t,\omega)=\exp\left\{\int_{0}^{t}A^{\nu}(s, \omega)\,ds\right\}.
\end{equation*}
Notice that $\tilde{u}^{\nu}(t,\omega)$ lies on the random invariant manifold~$\widetilde{\mathcal{M}}_{L^{2}(D)}(\psi^{\nu}\theta_{t}\omega)$, that $z^{*\nu}$~is stationary, that by the assumption~(\ref{e:Lip})    $S^{\nu}(t,\omega)$~is  nonuniformly pseudo-hyperbolic~\cite{Cara}, and that $B^{\nu}(t,\omega)$~is tempered  and locally integrable in~$t$. Then by a similar discussion to that for random evolutionary equation~\cite{Cara}, we also can construct a random  invariant manifold. 
Then, by the same discussion above for the estimates~$\nu\|e^{-\eta t}\tilde{u}^{\nu}_{t}(t,\omega)\|_{L^{2}(D)}$  and the estimate~(\ref{e:nu-eta-t-bound}), the bound on~$\nu z_{t}^{*\nu}(t, \omega)$ in Appendix \ref{sec:App-A},   we have the estimate
\begin{equation*}
\nu\|e^{-\eta t}\tilde{w}^{\nu}_{t}(t,\omega)\|_{L^{2}(D)}\rightarrow 0\quad \text{almost surely for any } t\leq 0\,, \quad \text{as}\quad \nu\rightarrow 0\,,
\end{equation*}
which completes the proof.  
\end{proof}

Now we establish the following theorem on the relation between~$\overline{\mathcal{M}}_{E}^{\nu}(\omega)$ and~$\widetilde{\mathcal{M}}_{ER}(\psi^\nu\omega)$.

\begin{theorem}\label{thm:app1}
Suppose $f\in C^{1}(L^{2}(D), L^{2}(D))$ and $N>0$ large enough. Then for any $R>0$ 
\begin{equation*}
\lim_{\nu\rightarrow 0} {\rm dist}_{E}\left(\widetilde{\mathcal{M}}_{ER}(\psi^{\nu}\omega), \overline{\mathcal{M}}_{E}^{\nu}(\omega)\right)=0\,.
\end{equation*}
\end{theorem}

 
\begin{proof}
We adapt the discussion for the deterministic case~\cite{ChowLu89}.

Let  any element $(\tilde{u}_{0}, \tilde{u}^{\nu}_{t}(0, \tilde{u}_{0}))\in \widetilde{M}_{ER}(\psi^{\nu}\omega)$  with $\tilde{u}^{\nu}$ satisfying equation~(\ref{e:RHE-u-tilde}) with initial condition $\tilde{u}^{\nu}(0)=\tilde{u}_{0}$\,. 
Define
\begin{eqnarray*}
\tilde{U}^{\nu}(t)=\left(\tilde{u}^{\nu}(t),
\frac{1}{2}\tilde{u}^{\nu}(t)+\nu\tilde{u}^{\nu}_t(t)\right),
\end{eqnarray*}
then $\tilde{U}^{\nu}=(\tilde{u}^{\nu}, \tilde{v}^{\nu})$ satisfies
\begin{eqnarray*}
\dot{\tilde{U}}^{\nu}(t,\omega)=C\tilde{U}^{\nu}(t,\omega)+
\begin{bmatrix}
  0 \\
  f(\tilde{u}^{\nu}(t,\omega)+z^{*\nu}(t,\omega)) \\
\end{bmatrix}
+\begin{bmatrix}
   0 \\
  \nu \tilde{u}^{\nu}_{tt}(t,\omega) \\
 \end{bmatrix}.
\end{eqnarray*}
Let $\bar{U}^{\nu}\in \overline{\mathcal{M}}_{E}^{\nu}(\omega)$ be a solution of the \rds~(\ref{e:U-nu-REE}) and $0<\nu<\nu_0$\,. Let 
\begin{equation*}
\hat{U}^{\nu}(t,  \omega)=\tilde{U}^{\nu}(t, \omega)-\bar{U}^{\nu}(t, \omega).
\end{equation*}
Hence, $\hat{U}^{\nu}(t,  \omega)$ satisfies 
\begin{eqnarray*}
 \dot{\hat{U}}^{\nu}(t, \omega)
&=&C\hat{U}^{\nu}(t, \omega)+\begin{bmatrix}
   0 \\
  \nu \tilde{u}^{\nu}_{tt}(t, \omega) \\
 \end{bmatrix}\\&&+{}
\begin{bmatrix}
  0 \\
  f(\tilde{u}^{\nu}(t,  \omega)+z^{*\nu}(t,  \omega))-f(\bar{u}^{\nu}(t, \omega)+z^{*\nu}(t, \omega)) \\
\end{bmatrix}.
\end{eqnarray*}

Notice that by the Lipschitz property of~$f$ and by Lemma~\ref{lem:u-tt}, 
\begin{equation*}
\Delta \tilde{u}^{\nu}(t,  \omega)=\tilde{u}^{\nu}_{t}(t, \omega)-f(\tilde{u}^{\nu}(t, \omega)+z^{*\nu}(t, \omega))\in C_{\eta, L^{2}(D)}^{-}\,.
\end{equation*}
 Then by the interpolation between~$H^{2}(D)$ and~$L^{2}(D)$,
 \begin{equation*}
 \tilde{u}^{\nu}\in C_{\eta, H_{0}^{1}}^{-}\,.
 \end{equation*}
By Lemma~\ref{lem:u-tt} we have $\hat{U}^{\nu}\in C_{\eta,E}^-$, then by the construction of solution in~$C_{\eta, E}^{-}$,  
\begin{eqnarray*}
\hat{U}^{\nu}(t)&=&e^{Ct}P\hat{U}^{\nu}(0)
\\&&{}
+\int_0^t e^{C(t-s)}P\left\{
\begin{bmatrix}
  0 \\
  f(\tilde{u}^{\nu}+z^{*\nu})-f(\bar{u}^{\nu}+z^{*\nu}) \\
\end{bmatrix}
+\begin{bmatrix}
   0 \\
  \nu \tilde{u}^{\nu}_{tt} \\
 \end{bmatrix}\right\}ds\\&&
 +\int_{-\infty}^t e^{C(t-s)}Q\left\{
\begin{bmatrix}
  0 \\
  f(\tilde{u}^{\nu}+z^{*\nu})-f(\bar{u}^{\nu}+z^{*\nu}) \\
\end{bmatrix}
+\begin{bmatrix}
   0 \\
  \nu \tilde{u}^{\nu}_{tt} \\
 \end{bmatrix}\right\}ds.
\end{eqnarray*}
Since $PE=E_1$ is of finite dimension, we can choose $\bar{u}(0)$~and~$\bar{u}_t(0)$
such that $P\hat{U}(0)=0$. Hence,
\begin{align*}
&e^{-\eta t}\|\hat{U}^{\nu}(t)\|_E
\\&
\leq e^{-\eta
t}\int_t^0e^{\lambda_N^+(t-s)}\left\{\left\|
\begin{bmatrix}
  0 \\
  f(\tilde{u}^{\nu}+z^{*\nu})-f(\bar{u}^{\nu}+z^{*\nu}) \\
\end{bmatrix}\right\|_E
+\left\|\begin{bmatrix}
   0 \\
  \nu \tilde{u}^{\nu}_{tt} \\
 \end{bmatrix}\right\|_E\right\}\,ds
 \\&\quad{}
 +e^{-\eta t}\int_{-\infty}^t e^{\lambda_{N+1}^+(t-s)}\left\{\left\|
\begin{bmatrix}
  0 \\
  f(\tilde{u}^{\nu}+z^{*\nu})-f(\bar{u}^{\nu}+z^{*\nu}) \\
\end{bmatrix}\right\|_E
+\left\|\begin{bmatrix}
   0 \\
  \nu \tilde{u}^{\nu}_{tt} \\
 \end{bmatrix}\right\|_E\right\}\,ds
 \\
  &\leq 3L_f\int_t^0e^{(\lambda_N^+-\eta)(t-s)}\|
\hat{U}^{\nu}\|_{C_{\eta, E}^-}\,ds
+3L_f\int_{-\infty}^te^{(\lambda_{N+1}^+-\eta)(t-s)}\|
\hat{U}^{\nu}\|_{C_{\eta,E}^-}\,ds
\\&\quad{}
+\nu e^{-\eta
t}\int_t^0e^{\lambda_N^+(t-s)}\|\tilde{u}^{\nu}_{tt}\|\,ds
   +\nu e^{-\eta
t}\int_{-\infty}^te^{\lambda_{N+1}^+(t-s)}\|\tilde{u}^{\nu}_{tt}\|\,ds.
\end{align*}
Then by Lemma~\ref{lem:u-tt},  for $N$~large enough we have 
$\|\hat{U}^{\nu}\|_{C_{\eta, E}^{-}}\rightarrow 0$\,, as $\nu\rightarrow 0$\,.
Hence, $\|\hat{U}^{\nu}(0)\|\rightarrow 0$\,,  as $\nu\rightarrow 0$\,.
The proof is complete.
\end{proof}

Next we show the approximation of the random dynamics~$u^{\nu}$. For this  we define the random sets
\begin{equation*}
\overline{\mathcal{M}}^{\nu}_{L^{2}(D)}(\omega)=
\left\{ \bar{u}: \bar{U}=(\bar{u}, \bar{v})\in\overline{\mathcal{M}}^{\nu}_{E} (\omega)\quad \text{for some }   \bar{v}\in L^{2}(D)\right\}
\end{equation*}
and 
\begin{equation*}
\overline{\mathcal{M}}^{\nu}_{L^{2}(D),R}(\omega)=\left\{ \bar{u}\in\overline{\mathcal{M}}^{\nu}_{L^{2}(D)}(\omega): \bar{U}=(\bar{u}, \bar{v})=\xi+h^{\nu}(\xi, \omega),\    \|\xi\|_{L^{2}(D)}\leq R\right\}.
\end{equation*}
We next prove that  for small parameter $\nu>0$\,,  $\overline{\mathcal{M}}^{\nu}_{L^{2}(D), R}(\omega)$ is approximated by~$\widetilde{\mathcal{M}}_{L^{2}(D)}(\psi^{\nu}\omega)$.  
\begin{theorem}\label{thm:main}
Suppose $f\in C^{1}(L^{2}(D), L^{2}(D))$ and $N>0$ large enough. Then for any $R>0$ 
\begin{equation*}
\lim_{\nu\rightarrow 0} {\rm dist}_{L^{2}(D)}\left(\overline{\mathcal{M}}_{L^{2}(D), R}^{\nu}(\omega), \widetilde{\mathcal{M}}_{L^{2}(D)}(\psi^{\nu}\omega) \right)=0\,.
\end{equation*}
\end{theorem}
\begin{proof}
The proof is  similar to that of the deterministic case.  Let $\bar{U}_{0}=(\bar{u}_{0}, \bar{v}_{0})$ with $\bar{u}_{0}\in \overline{\mathcal{M}}^{\nu}_{L^{2}(D), R}(\omega)$ and $\bar{U}^{\nu}(t,  \omega)=(\bar{u}^{\nu}(t, \omega), \bar{v}^{\nu}(t,  \omega))$ be the solution of~(\ref{e:U-nu-REE})  with $\bar{U}^{\nu}(0)=\bar{U}(0)$. By the invariance of~$\overline{\mathcal{M}}^{\nu}_{E}(\omega)$,   
\begin{equation*}
\bar{u}^{\nu}_{t}=A\bar{u}^{\nu}+f(\bar{u}^{\nu}+z^{*\nu})-\nu \bar{u}^{\nu}_{tt}\,.
\end{equation*}
Let $\tilde{u}^{\nu}(t, \omega)$ be the solution of equation~(\ref{e:RHE-u-tilde}) on the random inertial manifold~$\widetilde{\mathcal{M}}_{L^{2}(D)}(\psi^{\nu}\omega)$ with $\tilde{u}^{\nu}(0)=\tilde{u}_{0}\in L^{2}(D)$.
Thus $\hat{u}^{\nu}(t, \omega)=\tilde{u}^{\nu}(t, \omega)-\bar{u}^{\nu}(t, \omega)$
satisfies 
\begin{eqnarray*}
\hat{u}^{\nu}_{t}(t, \omega)&=&A\hat{u}^{\nu}(t, \omega)+f(\tilde{u}^{\nu}(t, \omega)+z^{*\nu}(t, \omega))\\&&{}-f(\bar{u}^{\nu}(t, \omega)+z^{*\nu}(t, \omega))+\nu\bar{u}^{\nu}_{tt}(t,\theta_{-t}\omega).
\end{eqnarray*}
Notice that $\tilde{u}^{\nu}(t, \omega)\in C_{\eta, L^{2}(D)}^{-}$  and $\bar{U}^{\nu}\in\overline{\mathcal{M}}^{\nu}_{E}(\omega)$, we have $\hat{u}^{\nu}\in C_{\eta, L^{2}(D)}^{-}$\,. 
Then,  
\begin{eqnarray*}
\hat{u}^{\nu}(t,
 \omega)&=& e^{At}P_{N}\hat{u}^{\nu}(0)+\int_{0}^{t}e^{A(t-s)}P_{N}[f(\tilde{u}^{\nu}(s, \omega)+z^{*\nu}(s, \omega))
 \\&&{}\qquad-f(\bar{u}^{\nu}(s, \omega)+z^{*\nu}(s, \omega))+\nu\bar{u}_{tt}^{\nu}(s, \omega)]\,ds\\
&&{}+\int_{-\infty}^{t}e^{A(t-s)}(I-P_{N})[f(\tilde{u}^{\nu}(s, \omega)+z^{*\nu}(s, \omega))
\\&&{}\qquad-f(\bar{u}^{\nu}(s, \omega)+z^{*\nu}(s, \omega))+\nu\bar{u}_{tt}^{\nu}(s, \omega)]\,ds\,.
\end{eqnarray*}
We need an estimate on~$\nu\bar{u}^{\nu}_{tt}$.   Suppose we have the following expansion in the orthonormal basis of the eigenfunctions~$e_{k}$ of~$A$, $\bar{u}_{t}^{\nu}=\sum_{k=1}^{\infty}\bar{u}^{\nu}_{t,k} e_{k}$\,.
Then by integration by parts, 
\begin{eqnarray*}
&&\sup_{t\leq 0}e^{-\eta t}\left\|\int_{0}^{t}e^{A(t-s)}P_{N}\bar{u}^{\nu}_{tt}(s)\,ds\right\|_{L^{2}(D)}\\&=&\sup_{t\leq 0}e^{-\eta t}\left\|P_{N}\bar{u}_{t}^{\nu}(t)-e^{At}P_{N}\bar{u}^{\nu}_{t}(0)-\sum_{k=1}^{N}k^{2}\int_{0}^{t}e^{-k^{2}(t-s)}\bar{u}_{t,k}^{\nu}(s) e_{k}\,ds\right\|_{L^{2}(D)}\\
&\leq&\sup_{t\leq 0}e^{-\eta t}\|P_{N}\bar{u}^{\nu}_{t}(t)\|_{L^{2}(D)}+\sup_{t\leq 0}e^{-(\eta +N^{2})t}\|P_{N}\bar{u}^{\nu}_{t}(0)\|_{L^{2}(D)}
\\&&{}
 +\sup_{t\leq 0}e^{-\eta t}\left\|\sum_{k=1}^{N}k^{2}\int_{0}^{t}e^{-k^{2}(t-s)}\bar{u}_{t,k}^{\nu}(s) e_{k}\,ds\right\|_{L^{2}(D)}\,.
\end{eqnarray*}
For the last term in the above equation,  we consider its square as
\begin{eqnarray*}
&&\sup_{t\leq 0}e^{-2\eta t}\left\|\sum_{k=1}^{N}k^{2}\int_{0}^{t}e^{-k^{2}(t-s)}\bar{u}_{t,k}^{\nu}(s) e_{k}\,ds\right\|^{2}_{L^{2}(D)}\\
&=&\sup_{t\leq 0}\sum_{k=1}^{N} \left[k^{2}\int_{0}^{t}e^{-k^{2}(t-s)}e^{-\eta(t-s)}e^{-\eta s}\|\bar{u}_{t,k}^{\nu}(s)\|_{L^{2}(D)}\,ds\right]^{2}\\
&\leq &\sup_{t\leq 0}\sum_{k=1}^{N} \left[k^{2}\int_{0}^{t}e^{-k^{2}(t-s)}e^{-\eta(t-s)}\,ds\right]^{2}\|\bar{u}^{\nu}_{t,k}\|^{2}_{C_{\eta,\R}}\\
&\leq&\|P_{N}\bar{u}^{\nu}_{t}\|^{2}_{C_{\eta, L^{2}(D)}^{-}} \,,
\end{eqnarray*}
where we use  that $\eta<-N^{2}$. Then  
\begin{equation}\label{e:P-N-baru}
\sup_{t\leq 0}e^{-\eta t}\left\|\int_{0}^{t}e^{A(t-s)}P_{N}\bar{u}^{\nu}_{tt}(s)\,ds\right\|_{L^{2}(D)}\leq 3\|P_{N}\bar{u}^{\nu}_{t}\|_{C_{\eta, L^{2}(D)}^{-}}\,.
\end{equation}
For the higher modes of~$\bar{u}^{\nu}
_{t}$, similarly we have 
\begin{eqnarray}
&&\sup_{t\leq 0}e^{-\eta t}\left\|\int_{-\infty}^{t}e^{A(t-s)}(I-P_{N})\bar{u}^{\nu}_{tt}(s)\,ds\right\|_{L^{2}(D)}\nonumber\\
&\leq&3\sup_{t\leq 0}e^{-\eta t}\|(I-P_{N})\bar{u}^{\nu}_{t}(t)\|_{L^{2}(D)} \nonumber\\
&\leq&3\|(I-P_{N})\bar{u}^{\nu}_{t}\|_{C_{\eta, L^{2}(D)}^{-}}  \,.\label{e:Q-N-baru}
\end{eqnarray}
Since $\bar{U}^{\nu}_{t}=(\bar{u}^{\nu}_{t}, \bar{v}^{\nu}_{t})\in C_{\eta, E}^{-}$, and by the same discussion as that for Theorem~\ref{thm:app1}, we have 
\begin{equation}
|\hat{u}^{\nu}|_{C_{\eta, L^{2}(D)}^{-}}\rightarrow 0\,,\quad \text{as}\quad \nu\rightarrow 0\,.
\end{equation}
This completes the proof.
\end{proof}

\begin{remark}
As stationary solutions lie on a random invariant manifold,  the above approximation for a random invariant manifold 
implies that the distribution of stationary solutions 
to both \textsc{swe}~(\ref{e:SWE1}) and \textsc{she}~(\ref{e:SHE1}) coincide with each other.  This coincidence was also shown by Cerrai and Freidlin~\cite{CF06} under certain conditions.
\end{remark}


\appendix
\section{Stationary solutions of linear SWEs and estimates}\label{sec:App-A}
We give some estimates on the stationary solution$(z^{*\nu}, z^{*\nu}_{t})$ to the linear \swe~(\ref{e:LSWEs}) and the stationary solution~$z^{*}$ to linear hear equation~(\ref{e:eta}). 

The following theorem is classical~\cite{CF06}.
\begin{theorem}
The stationary solution~$z^{*\nu}$ is Gaussian with normal distribution~$\mathcal{N}\left(0, \tfrac12 A^{-1}Q\right)$ in~$L^{2}(D)$, which is also the distribution of~$z^{*}$ in the space~$L^{2}(D)$. 
\end{theorem}
We consider~$z^{*\nu}$ and~$z^{*}$ on the canonical probability space~$(\Omega_{0}, \mathcal{F}_{0}, \mathbb{P})$ and the Wiener shift~$\{\theta_{t}\}_{t\in\R}$\,.  Then as   stationary solutions to  stochastic equations,  we  write $z^{*\nu}(t)=z^{*\nu}(t,\omega)=z^{*\nu}(\theta_{t}\omega)$ and $z^{*}(t)=z^{*}(t,\omega)=z^{*}(\theta_{t}\omega)$. 
\begin{theorem}
The processes~$z^{*\nu}(t,\omega)$ and~$z^{*}(t,\omega)$  satisfy 
\begin{equation*}
\lim_{t\rightarrow\pm\infty}\frac{1}{t}\|z^{*\nu}(t,\omega)\|_{L^{2}(D)}=\lim_{t\rightarrow\pm\infty}\frac{1}{t}\|z^{*}(t,\omega)\|_{L^{2}(D)}=0
\end{equation*}
for almost all $\omega\in\Omega$\,.
\end{theorem}
\begin{proof}
The proof is the same as that for scalar systems~\cite{DuanLuSch03}.
\end{proof}

We need an estimate on~$\nu z^{*\nu}_{t}$\,.  Since $\nu^{2}\mathbb{E}\|z^{*\nu}_{t}(t)\|^{2}=\nu \operatorname{Tr} Q\rightarrow 0$\,, as $\nu\rightarrow 0$\,,
then for almost all $\omega\in\Omega$
\begin{equation}\label{e:nu-eta-t-bound}
\nu z^{*\nu}_{t}(t,\omega)\rightarrow 0\quad \text{as}\quad \nu\rightarrow 0\,.
\end{equation}

%
%
%

\section{Rohlin's classification}\label{sec:Rohlin}

\begin{definition}
Two random variables~$X$ and~$Y$ defined on a same probability space~$(\Omega, \mathcal{F}, \mathbb{P})$  are called equivalent if and only if there is a measurable preserving map $\psi:\Omega\rightarrow\Omega$ such that $X(\psi\omega)=Y(\omega)$
for almost all $\omega\in\Omega$\,.
\end{definition}
If~$X$ is equivalent to~$Y$, then $X$ has same distribution as that of~$Y$.  
 Rohlin's result on the classification of the homomorphisms of a Lebesgue space~\cite{Rohlin49} gives an inverse result.  
 
 Recall that a homomorphism~$\psi$ from probability space~$(\Omega_{1}, \mathcal{F}_{1}, \mathbb{P}_{1})$ to probability space~$(\Omega_{2}, \mathcal{F}_{2}, \mathbb{P}_{2})$ is a measurable mapping  such that $\psi\mathbb{P}_{1}=\mathbb{P}_{2}$\,. If $\psi$~is measurably invertible, then $\psi$~is an isomorphism. A 
  probability space~$(\Omega, \mathcal{F}, \mathbb{P})$ is a Lebesgue space~\cite[Appendix~A]{Arn} if this probability space is  isomorphic to a probability space  which is the disjoint union of an at most countable (possibly empty) set~$\{\omega_{1}, \omega_{2}, \ldots \}$ of points each of positive measure and the space~$([0, s), \mathcal{B}, \lambda)$, where $\mathcal{B}$ is the $\sigma$-algebra of Lebesgue measurable subsets of the interval~$[0, s)$ and $\lambda$~is the Lebesgue measure. Here $s=1-\sum p_{n}$ where $p_{n}$~is the measure of the point~$\omega_{n}$.   For a measure space, the signature is the mass of its non-atomic part plus the non-increasing sequence of the weights of its atoms. 
  
 Rohlin's classification theorem on the homomorphisms of Lebesgue space states the following~\cite{Rohlin49}.
 \begin{theorem} 
 A homomorphisms of Lebesgue space is determined by the signature of the quotient measure space  and the signatures of the condition measure spaces associated with the homomorphism. 
 \end{theorem}
 Then we have the following corollary.
 \begin{coro}
 Random variables~$X$ and~$Y$ are equivalent if and only if for almost all values taken by these variables, the condition measure spaces are isomorphic; that is, they have the same  signature. 
 \end{coro}
One special case  for the above condition measure spaces having the same signature is when almost all conditional measures are purely non-atomic. 

  The canonical  probability space~$(\Omega_{0}, \mathcal{F}_{0}, \mathbb{P}_{0})$ is a Lebesgue space~\cite{Arn}.   
Now we consider~$\eta^{*\nu}(\omega)$ and~$\eta^{*}(\omega)$ which have the same distribution  on the same probability space~$(\Omega_{0}, \mathcal{F}_{0}, \mathbb{P}_{0})$.  Moreover, the distribution is Gaussian, so almost all conditional measures are purely non-atomic. Then by the above corollary, there is a measure preserving mapping $\psi^{\nu}:\Omega_{0}\rightarrow\Omega_{0}$ such that 
\begin{equation}
\eta^{*}(\psi^{\nu}\omega)=\eta^{*\nu}(\omega).
\end{equation}

\paragraph{Acknowledgements} This research was supported by the
NSFC grant No.~10901083, the ZiJin Foundation of Nanjing University of Science and Technology and
by the Australian Research Council grants DP0774311 and DP0988738. The first author was also supported by the CSC to work at the University of Adelaide.

\end{document}